\title{Asymptotic distributions and chaos for the supermarket model}
\date{}

\documentclass[10pt]{article}

\usepackage{amssymb}
\usepackage{latexsym}

\newcommand{\m}[1]{\marginpar{\tiny{#1}}}

\newenvironment{proof}{\noindent {\bf
Proof}.}{\proofbox\par\smallskip\par}
\newenvironment{proofof}[1]{\noindent {\bf
Proof of #1}.}{\proofbox\par\smallskip\par}

\newtheorem{theorem}{Theorem}[section]
\newtheorem{lemma}[theorem]{Lemma}
\newtheorem{proposition}[theorem]{Proposition}

\newtheorem{corollary}[theorem]{Corollary}

\newcommand{\halmos}{\rule{1ex}{1.4ex}}
\newcommand{\proofbox}{\hspace*{\fill}\mbox{$\halmos$}}

\newcommand{\reals}{\mathbb{R}}
\newcommand{\nats}{\mathbb{N}}

\newcommand{\ints}{\mathbb{Z}}

\newcommand{\ind}{\mathbb{I}}

\newcommand{\E}{\mbox{\bf E}}
\newcommand{\pr}{\mbox{\bf Pr}}

\newcommand{\dtv}{d_{\mbox{\tiny TV}}}
\newcommand{\dw}{d_{\mbox{\tiny W}}}

\newcommand{\none}[1]{\parallel\!{#1}\!\parallel_1}
\newcommand{\ninf}[1]{\parallel\!{#1}\!\parallel_{\infty}}

\author{Malwina J Luczak \thanks{supported in part by the Nuffield Foundation}\\
Department of Mathematics, London School of Economics,\\ Houghton
Street, London WC2A 2AE, UK.\\ e-mail: m.j.luczak@lse.ac.uk
  \and Colin McDiarmid\\
  Department of Statistics, University of Oxford,\\
  1 South Parks Road, Oxford OX1 3TG, UK.\\
  e-mail: cmcd@stats.ox.ac.uk}

\begin{document}

\maketitle

\begin{abstract}
In the supermarket model there are $n$ queues, each with a unit
rate server. Customers arrive in a Poisson process at rate
$\lambda n$, where $0<\lambda<1$.  Each customer chooses $d \geq
2$ queues uniformly at random, and joins a shortest one.

It is known that the equilibrium distribution of a typical queue
length converges to a certain explicit limiting distribution as $n
\to \infty$. We quantify the rate of convergence by showing that
the total variation distance between the equilibrium distribution
and the limiting distribution is essentially of order $n^{-1}$;
and we give a corresponding result for systems starting from quite
general initial conditions (not in equilibrium). Further, we
quantify the result that the systems exhibit chaotic behaviour: we
show that the total variation distance between the joint law of a
fixed set of queue lengths and the corresponding product law is
essentially of order at most $n^{-1}$.
\end{abstract}

\vfill



{Submitted 21 August 2005, accepted 10 January 2007.}

   \section{Introduction}
   We consider the following well-known scenario, often referred to
   as the `supermarket model'~\cite{lm04,ln03,m96a,m96b,m02}. Let $d$
   be a fixed integer at least 2. Let $n$
   be a positive integer and suppose that there are $n$
servers, each with a separate queue. Customers arrive in a Poisson
process at rate $\lambda n$, where $\lambda \in (0,1)$ is a
constant. Upon arrival each customer chooses $d$ servers uniformly
at random with replacement, and joins a shortest queue amongst
those chosen. If there is more than one chosen server with a
shortest queue, then the customer goes to the first such queue in
her list of $d$. Service times are independent unit mean
exponentials, and customers are served according to the first-come
first-served discipline.

   Recent work on the supermarket model
   includes~\cite{g99,g00,g04,l03,lm04,ln03,vdk96}. The
   survey~\cite{m02} gives several applications and related results: an
   important application is to load balancing when computational tasks
   are dynamically assigned to servers in a large array. It is shown
   in~\cite{g99,g00} that the system exhibits {\em
   propagation of chaos} given
a suitable initial state, and in particular if it is in
equilibrium. This means that the paths of members of any fixed
finite subset of queues are asymptotically independent of one
another, uniformly on bounded time intervals. This result implies
a law of large numbers for the time evolution of the proportion of
queues of different lengths, more precisely for the empirical
measure on path space~\cite{g99,g00}. In particular for each fixed
positive integer $k_0$, as $n$ tends to infinity the proportion of
queues with length at least $k_0$ at time $t$ converges weakly to
a function $v_t(k_0)$, where $v_t(0)=1$ for all $t \ge 0$ and
$(v_t(k): k \in \nats)$ is the unique solution to the system of
differential equations
\begin{equation}
\label{eqn.lln}
  {{dv_t(k)} \over {dt}} = \lambda (v_t(k-1)^d-v_t(k)^d) - (v_t(k) -
v_t(k+1))
\end{equation}
for $k \in \nats$ (see~\cite{vdk96}). Here one assumes appropriate
initial conditions $v_0=(v_0(k): k \in \nats)$ such that $1 \ge
v_0(1) \ge v_0(2) \geq \cdots \ge 0$, and $v_0 \in l_1$. Further,
again for a fixed positive integer $k_0$, as $n$ tends to
infinity, in the equilibrium distribution the proportion of queues
with length at least $k_0$ converges in probability to
$\lambda^{(d^{k_0}-1)/(d-1)}$, and thus the probability that a
given queue has length at least $k_0$ also converges to
$\lambda^{(d^{k_0}-1)/(d-1)}$.

Recent results in~\cite{lm04} include rapid mixing and two-point
concentration for the maximum queue length in equilibrium.

The main contribution of the present paper is to give quantitative
versions of the convergence results for the supermarket model
mentioned above, and to extend them to hold uniformly over all
times. We rely in part on combinatorial techniques developed
in~\cite{lm03,lm04}.

   For each time $t \geq 0$ and each $j=1,\ldots,n \ $ let $X^{(n)}_t(j)$
   denote the number of customers in queue $j$,
   always including the customer currently being served if there is
   one. We shall keep the superscript `$n$' in the notation in this
   section, but then usually drop it in later sections.
   We make the standard assumption of right-continuity of the sample
   paths. Let $X^{(n)}_t$ be the {\em queue-lengths vector}
   $(X^{(n)}_t(1),\ldots, X^{(n)}_t(n)) \in \ints_+^n$, where
   $\ints_+^n$ denotes the set of all $n$-vectors with components taking
   non-negative integer values.  Note that the
   $l_1$-norm $\none{X_t}$ of $X_t$ is the total number of customers
   present at time $t$, and the $l_{\infty}$-norm $\ninf{X_t}$ is the
   maximum queue length.

   For a given positive integer $n$, the $n$-queue process
   $(X^{(n)}_t)$ is an ergodic
   continuous-time Markov chain. Thus there is a unique
   stationary distribution ${\bf \Pi}^{(n)}$ for the
   vector $X^{(n)}_t$; and, whatever the distribution of the
   starting state, the distribution of the vector
   $X^{(n)}_t$ at time $t$ converges to ${\bf \Pi}^{(n)}$ as
   $t \rightarrow \infty$.
   As already noted in~\cite{g99,g00} (and easily verified),
   the distribution ${\bf \Pi}^{(n)}$ is {\it exchangeable}, that
   is invariant under permuting the co-ordinates.
   We shall usually write $Y^{(n)}_t$ to denote the queue-lengths vector
   in equilibrium: \m{:} we drop the subscript $t$ when no explicit reference to
   a particular time is needed.

   The probability law of a random element $X$ will
   be denoted by ${\cal L}(X)$. The {\em total variation distance}
   between two probability distributions $\mu_1$ and $\mu_2$ is
   defined by $\dtv (\mu_1,\mu_2)= \sup_A | \mu_1(A) - \mu_2(A)|$
   where the supremum is over all measurable sets $A$ in the underlying
   measurable space (see also the start of Section~\ref{sec.prelim}).
   Also, given a vector $v=(v(k): k=0,1,\ldots)$ such that
   \begin{equation} \label{decr}
   1 = v(0) \ge v(1) \ge v(2) \ge \ldots \ge 0 \; \mbox{ and } \;
   v(k) \to 0 \mbox{ as } k \to \infty,
   \end{equation}
   let ${\cal L}_v$ denote the law of a random variable $V$
   taking non-negative integer values, where
   $\pr(V \geq k) = v (k)$ for each $k=0,1,2,\ldots $. In fact,
   throughout this paper we shall work only with vectors $v \in l_1$.

   Finally, throughout we use the asymptotic notations
   $O(),\Omega ( ), o(),\omega()$ in a standard way,
   to describe the behaviour of functions depending on the number
   of servers $n$ as $n$ tends to infinity; for instance
   $f(n) = \Omega(g(n))$ means that, for some constants $c > 0$ and $n_0$,
   we have
   $f(n) \geq c g(n)$ for all $n \geq n_0$.

   We may now state four main results, two
   concerning approximating the distribution of a single typical
   queue length and two concerning collections of queues and chaos.
   These will be proved in the following sections, where we also
   present some further results.

   %%%%%%%%%%%%%%%%%%%%%%%%%%%%%%%%%%%%%%%%%%%%%%%%%%%%%%%%%%%%%%%

   \subsection{Single queues}

   We first consider the $n$-queue system in equilibrium, and
   investigate how close the distribution of a typical queue length
   $Y^{(n)}(1)$ is to the limiting distribution.
   Let ${\cal L}_{\lambda,d}$ denote the law ${\cal L}_v$ where
   $v(k) = \lambda^{(d^k-1)/(d-1)}$ for each $k=0,1,\ldots$. Note
   that $\pr(Y^{(n)}(1) \geq 1) = \lambda = v(1)$. It is known (and
   was mentioned earlier) that ${\cal L}(Y^{(n)}(1))$ tends to
   ${\cal L}_{\lambda,d}$ as $n \to \infty$: we now quantify this
   convergence.

   \begin{theorem} \label{thm.distql}
   For each positive integer $n$ let $Y^{(n)}$ be a queue-lengths
   $n$-vector in equilibrium, and consider the length $Y^{(n)}(1)$ of
   queue~1. Then \[\dtv ({\cal L}(Y^{(n)}(1)), {\cal L}_{\lambda,d})\]
   is of order $n^{-1}$ up to logarithmic factors.
   \end{theorem}
    In fact, we shall see that the above total variation distance is
   $o(n^{-1} \ln^3n)$ and is $\Omega(n^{-1})$.  Also, we shall deduce
   directly from Theorem~\ref{thm.distql}, together with a bound on the
   maximum queue length from~\cite{lm04} (given as~(\ref{eqn.ub3})
   below), the following:
   \begin{corollary} \label{cor.distql}
   For each positive integer $k$, the difference between the $k$th moment
   $\E[Y^{(n)}(1)^k]$ and the $k$th moment of ${\cal L}_{\lambda,d}$
   is of order $n^{-1}$ up to logarithmic factors.
   \end{corollary}

   Now we drop the assumption that the system is in equilibrium,
   and consider its behaviour from the beginning (that is,
   from time~0).  We state one theorem here,
   Theorem~\ref{cor.distq3}.
   More general though less digestible results (Theorems~\ref{thm.distq2}
   and~\ref{cor.distq2}) are stated and proved in Section~\ref{sec.distql},
   and Theorem~\ref{cor.distq3} will follow easily from
   Theorem~\ref{cor.distq2}.
   We assume in Theorem~\ref{cor.distq3} that the initial
   queue lengths are iid and not too large; and see that the law of
   a typical queue length is close to ${\cal L}_{v_t}$, where
   $v_t=(v_t(k): k \in \nats)$ is the solution to the
   system~(\ref{eqn.lln}) subject to the natural initial conditions.

   Given a queue-lengths vector $x$,
   that is a vector with non-negative integer co-ordinates,
   we let $u(k,x)$ be the proportion
   of queues with length at least~$k$.

   \begin{theorem} \label{cor.distq3}
   There is a constant $\varepsilon >0$
   such that the following holds.
   Let the random variable $X$ take non-negative values,
   and suppose that %
   $\E[e^{X/\varepsilon}]$ is finite.
   For each positive integer $n$, let $(X^{(n)}_t)$ be an $n$-queue
   process where the random initial queue lengths
   $X_0^{(n)}(1),\ldots,X_0^{(n)}(n)$ are iid
and stochastically at most $X$.
   Let $v^{(n)}_t = (v^{(n)}_t(k): k=0,1,2,\ldots)$ be the unique
   solution to the system~(\ref{eqn.lln}) subject to the initial
   conditions $v^{(n)}_0(k)= \E[u(k,X^{(n)}_0)]$ for $k=0,1,\ldots$.

   Then the law of $X^{(n)}_t$ is exchangeable for each $t \geq 0$, and
\[\sup_{t \ge 0} \dtv ({\cal L}(X^{(n)}_t(1)), {\cal L}_{v^{(n)}_t})
   = O(n^{-\varepsilon}).\]
   \end{theorem}

%%%%%%%%%%%%%%%%%%%%%%%%%%%%%%%%%%%%%%%%%%%%%%%%%%%%%%%%%%%%%%%%%%%%%%%%%%
%%%%%%%%%%%%%%%%%%%%%%%%%%%%%%%%%%%%%%%%%%%%%%%%%%%%%%%%%%%%%%%%%%%%%%%%%%

   \subsection{Collections of queues}

   The above results concern the distribution of a single queue
   length. We now consider collections of queues and propagation of
   chaos. The term ``propagation of chaos'' comes from statistical
   physics~\cite{kac}, and the original motivation was the evolution of
   particles in physical systems. The subject has since then received
   considerable attention culminating in the work of Sznitman~\cite{s91}.

Our results below establish chaoticity for the
   supermarket model.
   As before, we first discuss the equilibrium distribution.
   We see that for fixed $r$ the total variation distance between
   the joint law of $r$ queue lengths and the product law is at most
   $O(n^{-1})$, up to logarithmic factors.  More precisely and more generally
   we have:
   \begin{theorem} \label{thm.chaos1}
   For each positive integer $n$, let $Y^{(n)}$ be a queue-lengths
   $n$-vector in equilibrium. Then, uniformly over all positive
   integers $r \leq n$, the total variation distance between the
   joint law of $Y^{(n)}(1), \ldots, Y^{(n)}(r)$ and the product law
   ${\cal L}(Y^{(n)}(1))^{\otimes r}$
   is at most
   $O(n^{-1} \ln^2 n (2 \ln \ln n)^r)$;
   and the total variation distance between the
   joint law of $Y^{(n)}(1), \ldots, Y^{(n)}(r)$ and
   the limiting product law ${\cal L}_{\lambda,d}^{\otimes r}$
   is at most
   $O(n^{-1} \ln^2 n (2 \ln \ln n)^{r+1})$.
   \end{theorem}
   Note that since the distribution of $Y^{(n)}$ is exchangeable, for
   any distinct indices $j_1,\ldots,j_r$ the joint distribution of
   $Y^{(n)}(j_1), \ldots, Y^{(n)}(j_r)$ is the same as that of
   $Y^{(n)}(1), \ldots, Y^{(n)}(r)$.
   Also, if $d$ were 1 we would have `exact' independence of all
   queues in equilibrium. Further, note that the above result
   can yield a bound less than 1 only if $r < \ln n / \ln\ln\ln n$;
   but as long as $r =o( \ln n / \ln\ln\ln n)$ the bound is
   $O(n^{-1+o(1)})$.
   We shall mention in Section~\ref{sec.chaos} how this result relates to
   Sznitman's treatment of chaos in~\cite{s91}
     -- see in particular the inequalities~(\ref{eq.dtv})
     and~(\ref{eq.dtv-1}) below.\bigskip

   \smallskip

%%%%%%%%%%%%%%%%%%%%%%%%%%%%%%%%%%%%%%%%%%%%%%%%%%%%%%%%%%%%%%%%%%
%%%%%%%%%%%%%%%%%%%%%%%%%%%%%%%%%%%%%%%%%%%%%%%%%%%%%%%%%%%%%%%%%%

   Now we drop the assumption that the system is in equilibrium,
   and consider its behaviour from time~0.  We state one theorem here,
   Theorem~\ref{cor.chaos2}.
   More general though less digestible results (Theorems~\ref{thm.chaos2}
   and~\ref{cor.chaos1}) are stated and proved in Section~\ref{sec.chaos},
   and Theorem~\ref{cor.chaos2} will follow easily from
   Theorem~\ref{cor.chaos1}.
   As with Theorem~\ref{cor.distq3} earlier, we assume in
   Theorem~\ref{cor.chaos2} that the initial
   queue lengths are iid and not too large:
   now we see that the joint law of
   $r$ queue lengths is close to the product law, that is we have
   chaotic behaviour, uniformly for all times.

   Given an $n$-queue process $(X_t^{(n)})$ and a positive integer $r \leq n$,
   let ${\cal L}_t^{(n,r)}$ denote the joint law of
   $X_t^{(n)}(1),\ldots, X_t^{(n)}(r)$, and let $\tilde{\cal L}_t^{(n,r)}$
   denote the product law ${\cal L}(X_t^{(n)}(1))^{\otimes r}$
   of $r$ independent copies of $X_t^{(n)}(1)$.
   \begin{theorem} \label{cor.chaos2}
   Let the random variable $X$ take
   non-negative values, and suppose that the moment
   generating function $\E[e^{s_0X}]$ is finite for some $s_0>0$.
   For each  integer $n \geq r$, let $(X^{(n)}_t)$ be an $n$-queue
   process where the random initial queue lengths $X_0^{(n)}(1),\ldots,X_0^{(n)}(n)$
   are iid and stochastically at most $X$.
   Then the law of $X^{(n)}_t$ is exchangeable for each $t \geq 0$,
   and uniformly over positive integers $r \leq n$
   \[ \sup_{t \geq 0} \ \dtv \left( {\cal L}_t^{(n,r)}, \tilde{\cal L}_t^{(n,r)}
   \right) = O \left( n^{-1} (\ln n)^{r+2} \right).\]
   \end{theorem}

   We remark also that an $O(n^{-1})$ upper bound on the total variation
   distance between the law of a finite $r$-tuple of queues and the
   product law is known on fixed time
   intervals for iid
   initial queue lengths~-~see~\cite{g00},
   Theorem 3.5. However, the bound in~\cite{g00} grows exponentially
   in time, and does not extend to the equilibrium case. We are not
   aware of any earlier equilibrium bounds or time-uniform bounds
   like those given in Theorems~\ref{thm.chaos1},~\ref{cor.chaos2},
   ~\ref{thm.chaos2} and~\ref{cor.chaos1}.

   Finally, let us mention at this point that connections between
   rapid mixing, concentration of measure and chaoticity have been
   observed earlier in various other contexts. The reader is referred
   to~\cite{gz02,l01,s91,v05} and references therein for more information.

%%%%%%%%%%%%%%%%%%%%%%%%%%%%%%%%%%%%%%%%%%%%%%%%%%%%%%%%%%%%%%%%%%%%%%%%%%%%%%%%%%%%
%%%%%%%%%%%%%%%%%%%%%%%%%%%%%%%%%%%%%%%%%%%%%%%%%%%%%%%%%%%%%%%%%%%%%%%%%%%%%%%%%%%%%%%%%%%

   \section{Preliminaries}
   \label{sec.prelim}

   This section contains some definitions and some results from~\cite{lm04}
   needed in our proofs.
   We start with two lemmas which show that the
   supermarket model is rapidly mixing. Lemma~\ref{thm.stat} upper
   bounds the total variation distance and Lemma~\ref{thm.statW}
   upper bounds the Wasserstein distance between ${\cal L}(X_t)$ and
   the equilibrium distribution ${\bf \Pi}$.

   We say that a real-valued function $f$ defined on a subset $A$ of $\reals^n$ is
   1-Lipschitz if
   $|f(x)-f(y)| \leq \none{x-y}$ for all $x,y \in A$.
   Let us recall two equivalent definitions of the total variation distance
   $\dtv (\mu_1,\mu_2)$ between two probability distributions $\mu_1$ and $\mu_2$
   on $\ints_+^n$,
   and two corresponding definitions of the Wasserstein distance
   $\dw(\mu_1,\mu_2)$.
   We have $\dtv (\mu_1,\mu_2)= \inf \pr(X \neq Y)$, and
   $\dw(\mu_1,\mu_2) = \inf \E[ \none{X - Y}]$, where
   in each case the infimum is over all couplings of $X$ and $Y$ where
   ${\cal L}(X)= \mu_1$ and ${\cal L}(Y)=\mu_2$.
   Also,
   \[\dtv (\mu_1,\mu_2)=
   \frac12 \sup_{\ninf{\phi} \leq 1} |\int \phi \ d\mu_1 - \int \phi
   \ d\mu_2|\]
   where the supremum is over measurable functions $\phi$ with
   $\ninf{\phi} \leq 1$; and
   \[\dw(\mu_1,\mu_2)=
   \sup_{f \in {\cal F}_1} |\int f \ d\mu_1 - \int f \
   d\mu_2|\]
   where the supremum is over the set ${\cal F}_1$ of 1-Lipschitz functions
   $f$ on $\ints_+^n$.
   The total variation distance between the corresponding
   laws on $\ints_+^n$
   is always at most the Wasserstein distance: to see this, note
   for example that \mbox{$\ind_{X \neq Y} \le \none{X-Y}$}.

   Recall that $0<\lambda<1$ and $d \in \nats$ are fixed.
   \begin{lemma} \label{thm.stat}
   For each constant $c>0$ there exists a constant $\eta >0$ such that
   the following holds for each positive integer $n$.
   For each time $t \geq 0$ let
   \[ \delta_{n,t} = \pr(\none{X^{(n)}_0} > c n) + \pr(\ninf{X^{(n)}_0}
   > \eta t).\]
   Then
   \[ \dtv ({\cal L}(X^{(n)}_t),{\bf \Pi}) \leq
   n e^{-\eta t} + 2 e^{-\eta n} + \delta_{n,t}.\]
   \end{lemma}

   \begin{lemma} \label{thm.statW}
   For each constant $c>\frac{\lambda}{1-\lambda}$ there exists a
   constant $\eta >0$ such that the following holds for each positive
   integer $n$. Let $M^{(n)}$ denote the stationary maximum queue length.
   Consider any distribution of the initial queue-lengths vector
   $X^{(n)}_0$ such that $\none{X^{(n)}_0}$ has finite mean.
   For each time $t \geq 0$ let
   \[ \delta_{n,t} =
   2 \E[\none{X^{(n)}_0} {\bf 1}_{\none{X^{(n)}_0}>cn}]+ 2cn \
   \pr(\ninf{X^{(n)}_0} > \eta t).\]
   Then
   \[ \dw ({\cal L}(X^{(n)}_t), {\bf \Pi})
   \leq n e^{-\eta t} + 2c n \pr(M^{(n)}> \eta t)+
   2e^{-\eta n} + \delta_{n,t}.\]
   \end{lemma}

   We now introduce the natural coupling of $n$-queue
   processes $(X_t)$ with different initial states.
   Arrival times form a Poisson process at rate $\lambda n$, and
   there is a corresponding sequence of uniform choices of lists of
   $d$ queues.
   Potential departure times form a Poisson process at rate $n$, and
   there is a corresponding sequence of uniform selections of a queue:
   potential departures from empty queues are ignored.
   These four processes are independent.
   Denote the arrival time process by ${\bf T}$, the
   choices process by ${\bf D}$, the potential
   departure time process by $\tilde{\bf T}$ and the selection process
   by $\tilde{\bf D}$.

   Suppose that we are given a sequence of arrival times $\bf t$ with
   corresponding queue choices $\bf d$, and a sequence of potential
   departure times $\tilde{\bf t}$ with corresponding
   selections $\tilde{\bf d}$ of a queue (where all these times are
   distinct).
   For each possible initial queue-lengths vector $x \in \ints_+^n$
   this yields a deterministic queue-lengths process
   $(x_t)$ with $x_0=x$: let us write
   $x_t=s_t(x;{\bf t},{\bf d},\tilde{{\bf t}},\tilde{{\bf d}})$.
   Then for each $x \in \ints_+^n$, the process
   $(s_t(x;{\bf T},{\bf D},\tilde{{\bf T}},\tilde{{\bf D}}))$
   has the distribution of a queue-lengths process with initial state $x$.
   The following lemma from~\cite{lm04} shows that the coupling has
   certain desirable properties.
   Part~(c) is not stated explicitly in~\cite{lm04},
   but it follows easily from part~(b).
   Inequalities between vectors in the statement are understood to
   hold component by component.
   `Adding a customer' means adding an arrival time and
   corresponding choice of queues.
   \begin{lemma} \label{lem.detcoupling}
   Fix any 4-tuple ${\bf t},{\bf d},\tilde{{\bf t}},\tilde{{\bf d}}$
   as above, and for each $x \in \ints_+^n$ write $s_t(x)$ for
   $s_t(x;{\bf t},{\bf d},\tilde{{\bf t}},\tilde{{\bf d}})$. Then
   \begin{description}
   \item{(a)} for each $x,y \in \ints_+^n$, both
   $\parallel s_t(x) - s_t(y) \parallel_1$
   and $\parallel s_t(x) - s_t(y) \parallel_{\infty}$ are non-increasing.
   \item{(b)}
   if $0 \leq t < t'$ and $s_t(x) \leq s_t(y)$ then
   $s_{t'}(x) \leq s_{t'}(y)$;
   \item{(c)} if ${\bf t'}$ and ${\bf d'}$
   are obtained from ${\bf t}$ and ${\bf d}$
   by adding some extra customers then, for all $t \ge 0$, $s_t(x;{\bf
   t'},{\bf d'},\tilde{{\bf t}},\tilde{{\bf d}}) \ge s_t(x;{\bf t},{\bf
   d},\tilde{{\bf t}},\tilde{{\bf d}})$.
   \end{description}
   \end{lemma}

   Next, let us consider the equilibrium distribution,
   and note some upper bounds on the total number of customers in the
   system and on the maximum queue length established in~\cite{lm04}.

   \begin{lemma} \label{lem.march04}

   (a) For any constant $c > \frac{\lambda}{1-\lambda}$, there is a
   constant $\eta>0$ such that for each positive integer $n$, in
   equilibrium the queue-lengths process $(Y^{(n)}_t)$ satisfies
   \[\pr\left( \none{Y^{(n)}_t} > c n \right) \leq e^{-\eta n}\]
   for each time $t \geq 0$.\\

   (b) For each positive integer $n$, in equilibrium the maximum
   queue-length $M^{(n)}_t$ satisfies
   \[\pr\left( M^{(n)}_t \geq j \right) \leq n \lambda^j\]
   for each positive integer $j$ and each time $t \geq 0$.

   \end{lemma}

   We shall require an upper bound on the maximum length $M^{(n)}$
   of a queue in equilibrium, from Section 7 of~\cite{lm04}. Let
   $i^*=i^*(n)$ be the smallest integer $i$ such that
     ${\lambda^{{d^i-1}\over {d-1}}}< n^{-1/2}\ln^2 n$. Then
     $i^* =\ln \ln n/\ln d+O(1)$; and
     with probability tending to 1 as $n \to \infty$,
     if $d=2$ then $M^{(n)}$ is $i^*$ or $i^* +1$, and if $d \geq 3$ then
     $M^{(n)}$ is $i^*-1$ or $i^*$. The bound we need is that if
   $r=r(n)=O(\ln n)$ then
   \begin{equation} \label{eqn.ub3}
   \pr(M^{(n)} \geq i^*+1+r) = e^{-\Omega(r \ln n)}.
   \end{equation}

   Now we state some concentration of measure results for the
   queue-lengths process $(X^{(n)}_t)$.
   Let us begin with the equilibrium case, where we use the notation
   $Y^{(n)}$.
   Recall that ${\cal F}_1$ denotes the set of 1-Lipschitz functions on
   $\ints_+^n$.
   (We suppress the dependence on $n$.)
   \begin{lemma} \label{lem.conca}
   There is a constant $c > 0$ such that the following holds.
   Let $n \ge 2$ be an integer and consider the $n$-queue system in
   equilibrium.
   Then for each $f \in {\cal F}_1$ and each $u \geq 0$
   \[ \pr \left( |f(Y^{(n)}) - \E[f(Y^{(n)})]| \geq u \right)
   \leq n e^{-c u/ n^{\frac12}}.\]
   \end{lemma}

   Let $\ell(k,x)$ denote $|\{j:x(j) \geq k\}|$, the number of queues
   of length at least $k$. Thus $\ell(k,x) = n u(k,x)$. Tight
   concentration of measure estimates for the random variables
   $\ell(k,Y)$ may be obtained directly from the last lemma.

   \begin{lemma} \label{lem.concb}
   Consider the stationary queue-lengths vector $Y^{(n)}$, and
   for each non-negative integer $k$ let
   $\ell(k) = \E[\ell(k,Y^{(n)})]$.  Then for any constant $c>0$,
   \[\pr (\sup_{k} |\ell(k,Y^{(n)})- \ell(k)| \geq c n^{\frac12} \ln^2 n) =
   e^{-\Omega(\ln^2n)}.\]
   Also, there exists a constant $c >0$ such that
   \[\sup_k \pr (|\ell(k,Y^{(n)})- \ell(k)| \geq c n^{\frac12} \ln n)
   =o(n^{-2}).\]
   Furthermore, for each fixed integer $r \ge 2$
   \[\sup_k |\E[\ell(k,Y^{(n)})^r] - \ell(k)^r| = O(n^{r-1}\ln^2 n).\]
   \end{lemma}
   The first and third parts of this lemma form Lemma 4.2
   in~\cite{lm04}: the second part follows directly from the preceding
   lemma.
   We now present a time-dependent result, which
   will be essential in the proof of Theorem~\ref{thm.distq2}.
   \begin{lemma} \label{lem.conc2}
   There is a constant $c>0$ such that the following holds.
   Let $n \geq 2$ be an integer, let $f \in {\cal F}_1$,
   let $x_0 \in \ints_+^n$,
   and let $X^{(n)}_0 = x_0$ almost surely.
   Then for all times $t \geq 0$ and all $u > 0$,
   \begin{equation} \label{eqn.extra3}
   \pr (|f(X^{(n)}_t) - \E[f(X^{(n)}_t)] | \geq u) \leq
   n \ e^{-\frac{c u^2}{nt+u}}.
   \end{equation}
   \end{lemma}

   We shall also use the following extension of Bernstein's inequality,
   which will follow easily for example from
   Theorem 3.8 of~\cite{cmcd98}.
   \begin{lemma}
   \label{lem.bernstein}
   Let $f$ be a 1-Lipschitz function on $\reals^n$.
   Let ${\bf Z} = (Z_1,\ldots,Z_n)$ where $Z_1,\ldots, Z_n$ are
   independent real-valued random variables, such that $Z_j$ has
   variance at most $\sigma_j^2$ and range at most $b$, for each
   $j=1,\ldots,n$.  Let $v=\sum_{j=1}^{n} \sigma_j^2$.   Then for
   each $u > 0$
   \[ \pr(|f({\bf Z}) - \E[f({\bf Z})]| \geq u) \leq 2 \exp
   \left(-\frac{u^2}{2v + \frac23 bu} \right).\]
   \end{lemma}
   \begin{proof}
   Given fixed numbers $z_1,\ldots,z_{j-1}$ where $1 \leq j \leq n$,
   for each $z$ let
    \begin{eqnarray*} g(z)
    & = & \E[f(z_1,\ldots,z_{j-1},z,Z_{j+1},\ldots,Z_n)].
   \end{eqnarray*}
   Then the function $g$ is 1-Lipschitz, so the random variable
   $g(Z_j)$ has variance at most $\sigma_j^2$ and range at most $b$.
   Thus, in the terms of Theorem 3.8 of~\cite{cmcd98}, the `sum of
   variances' is at most $v$, and the `maximum deviation' is at
   most~$b$; and so we may use that theorem to deduce the result
   here.
  \end{proof}

   Let $u_t(k)=\E[u(k,X_t)]$, the expected proportion of
   queues of length at least $k$ at time $t$. Also let $u(k)$ denote
   $\E[u(k,Y)]$, the expected proportion of queues with at least $k$
   customers when the process is in equilibrium.
   Then $u(0)=1$ and $u(1)= \lambda$.
   Whatever the initial distribution of $X_0$, for
   each positive integer $k$,
   \begin{equation} \label{gen0}
   {{d u_t(k)}\over {dt}}=\lambda \left( \E[u(k-1,X_t)^d] -
   \E[u(k,X_t)^d] \right) -(u_t(k)-u_t(k+1));
   \end{equation}
   and for $Y$ in equilibrium
   \begin{equation} \label{gen0equil}
   0 = \lambda \left( \E[u(k-1,Y)^d] - \E[u(k,Y)^d] \right)
   -(u(k)-u(k+1)).
   \end{equation}
   In~\cite{lm04}, this last fact is used to show that $u(k)$ is
   close to $\lambda^{(d^{k}-1)/(d-1)}$;
   more precisely, for some constant $c >0$, for
   each positive integer $n$
   \begin{equation} \label{eqn.utight}
   \sup_k |u(k)-\lambda^{(d^{k}-1)/(d-1)}| \le c n^{-1} \ln^2 n.
   \end{equation}

%%%%%%%%%%%%%%%%%%%%%%%%%%%%%%%%%%%%%%%%%%%%%%%%%%%%%%%%%%%%%%%%%%%%%%%%%%%%%%%%%%%%%%%%
%%%%%%%%%%%%%%%%%%%%%%%%%%%%%%%%%%%%%%%%%%%%%%%%%%%%%%%%%%%%%%%%%%%%%%%%%%%%%%%%%%%%%

   \section{Distribution of a single queue length}
   \label{sec.distql}

   \subsection{Equilibrium case}

   In this subsection we prove
    Theorem~\ref{thm.distql}.
    Let us note that the equilibrium distribution ${\bf \Pi}$ is
   exchangeable, and thus all queue lengths are identically distributed.

   We begin by showing that (a) in equilibrium the total variation
   distance between the marginal distribution of a given queue length
   and the limiting distribution ${\cal L}_{\lambda,d}$ is small; and
   (b) without assuming that
   the system is in equilibrium, a similar result holds after a
   logarithmic time, given a suitable initial distribution. [In fact,
   the result (a) will follow easily from (b), since by
   Lemma~\ref{lem.march04}
   and~(\ref{eqn.ub3}), if $X_0$ is in equilibrium and we set
   $c_0 > \lambda/(1-\lambda)$
   then the quantity $\delta_n$ in Proposition~\ref{prop.distql}
   below is $o(n^{-1})$.]
   Part (a) includes the upper bound part of Theorem~\ref{thm.distql} above.

   \begin{proposition} \label{prop.distql}

   \begin{description}
   \item{}
   \item{(a)}
   The equilibrium length $Y^{(n)}(1)$ of queue 1 satisfies
   \[\dtv ({\cal L}(Y^{(n)}(1)), {\cal L}_{\lambda,d}) =
   O \left( n^{-1} \ln^2 n \ \ln\ln n \right).\]
   \item{(b)} For any constant $c_0 >0$ there is a constant $c$ such that
   the following holds.
   Let $\delta_n = \pr(\none{X^{(n)}_0} > c_0 n) + \pr(\ninf{X^{(n)}_0} >
   c_0 \ln n)$.
   Then uniformly over $t \geq c \ln n$
   \[\max_{j=1,\ldots,n} \dtv ({\cal L}(X^{(n)}_t(j)), {\cal L}_{\lambda,d})
   \leq \delta_n + O(n^{-1} \ln^2 n \ \ln\ln n).\]
   \end{description}
   \end{proposition}
Let us note here that better bounds may be obtained in the simple
case $d=1$. Here the $n$ queues behave
   independently, and ${\cal L}(Y(1)) = {\cal L}_{\lambda,d}$ for each
   $n$. The arrival rate at each queue is always $\lambda$,
   regardless of the state of all the other queues. Then it follows
   from the proof of Lemma~2.1 in~\cite{lm04}
   (which is stated as Lemma~\ref{thm.stat} here)
     that we can drop the term
   \mbox{$\pr(\none{X_0} > c_0 n)$}
   in the mixing bound of that lemma, and
   so part (b) of Proposition~\ref{prop.distql}
   holds with the bound $\delta_n + O(n^{-1} \ln^2 n
   \ln\ln n)$ replaced by $\pr(\ninf{X_0} > c_0 \ln n) + O(n^{-K})$ for any
   constant $K$.
   \smallskip

   \begin{proofof}{Proposition~\ref{prop.distql}}
   Since the distribution of $Y$ is exchangeable,
   $\pr (Y(1) \geq k)= u(k)$ for each non-negative integer $k$.
   Note that $u(0)-u(1)=1-\lambda$. Part (a) now follows
easily from~(\ref{eqn.utight}), since
\begin{eqnarray*}
\dtv ({\cal L}(Y(1)), {\cal L}_{\lambda,d}) &=& {1 \over 2}
\sum_{k=1}^{\infty} |u(k)-u(k+1) - \lambda^{1+d+\cdots
+d^{k-1}} + \lambda^{1+d+\cdots +d^k}|\\
& \leq & \sum_{k=1}^{k_0 +1} |u(k) - \lambda^{1+d+\cdots
+d^{k-1}}| + u(k_0 +1) +  \lambda^{1+d+\cdots +d^{k_0}},
\end{eqnarray*}
for any positive integer $k_0$. But if $k_0 \geq \ln\ln n / \ln d
+c$, where $c= - \ln \ln(1/\lambda)/\ln d$ then
$\lambda^{1+d+\cdots +d^{k_0}} \leq n^{-1}$, and so also $u(k_0
+1) = O(n^{-1} \ln^2n)$.

   For part (b) note that
   \begin{eqnarray*}
   \dtv ({\cal L}(X_t(j)), {\cal L}_{\lambda,d})
   & \le & \dtv ({\cal L}(X_t(j)), {\cal L}(Y(j))) +
   \dtv ({\cal L}(Y(j)), {\cal L}_{\lambda,d}) \\
   & \leq & \dtv ({\cal L}(X_t), {\cal L}(Y)) +
   \dtv ({\cal L}(Y(1)),{\cal L}_{\lambda,d}).
   \end{eqnarray*}
   Thus, by  part (a) and by Lemma~\ref{thm.stat}, there exists a
   constant $\eta = \eta (c_0)$ such that for each time
   $t \ge \eta^{-1} c_0 \ln n$
   \[ \max_j \dtv ({\cal L}(X_t(j)), {\cal L}_{\lambda,d}) \le
   \delta_n + n e^{-\eta t} + 2 e^{-\eta n} + O(n^{-1} \ln^2 n \ \ln \ln n),\]
   and the result follows with $c_1=\max \{2,c_0\} \ \eta^{-1}$.
   \end{proofof}

\medskip

We now show that the $O \left( n^{-1} \ln^2 n \ \ln\ln n \right)$
upper bound on the total variation distance between the
equilibrium distribution of a given queue length and ${\cal
L}_{\lambda,d}$ in Proposition~\ref{prop.distql} (a) is fairly
close to optimal. The next lemma will complete the proof of
Theorem~\ref{thm.distql}.
   \begin{lemma} \label{lem.distq2}
   For an $n$-queue system in equilibrium,
   the expected proportion $u(2)$ of queues of length
   at least 2 satisfies  $u(2) \geq \lambda^{d+1} + \Omega(n^{-1})$. Hence
   \[\dtv ({\cal L}(Y^{(n)}(1)), {\cal L}_{\lambda,d}) =\Omega \left(
   n^{-1}\right).\]
   \end{lemma}
\begin{proof}
Let $F_t= \ell (1,Y_t)$ be the number of non-empty queues at time
$t$, and write $F$ for $F_1$. We shall show that the variance of
$F$ is $\Omega(n)$, and from that we shall complete the proof
quickly.

Recall that we model departures by a Poisson process at rate $n$
(giving potential departure times) together with an independent
selection process that picks a uniformly random queue at each
event time of this process. If the queue selected is nonempty,
then the customer currently in service departs; otherwise nothing
happens.

Let $Z$ be the number of arrivals in $[0,1]$. By the last part of
Lemma~\ref{lem.detcoupling} we have the monotonicity result that
for all non-negative integers $x$ and $z$,
\begin{equation} \label{eqn.mono}
\pr(F \leq x | Z=z) \geq \pr(F \leq x | Z=z+1).
\end{equation}

Let the integer $x=x(n)$ be a conditional median of $F$ given that
$Z = \lfloor \lambda n \rfloor$. (It is not hard to see that $x =
\lambda n + O((n \ln n)^{\frac12})$ but we do not use this here.)
Since $\pr(F \leq x | Z = \lfloor \lambda n \rfloor) \geq \frac12$
we have by~(\ref{eqn.mono}) that
\begin{eqnarray*}
\pr(F \leq x) & \geq & \pr(F \leq x | Z \leq \lambda n) \pr(Z \leq
\lambda n)\\
& \geq & \pr(F \leq x | Z = \lfloor \lambda n \rfloor) \pr(Z \leq \lambda n)\\
& \geq & \frac14 +o(1).
\end{eqnarray*}
We shall find a constant $\delta>0$ such that $\pr(F \geq x+
\delta n^{\frac12}) \geq \frac1{14} +o(1)$, which will show that
the variance of $F$ is $\Omega(n)$ as required.

Let $A$ be the event that $Z \geq \lambda n + (\lambda
n)^{\frac12}$.  Then, by the central limit theorem, $\pr(A) \geq
\frac17 + o(1)$. We shall need to condition on $A$.

If $Z \leq \lambda n$ then all the customers arriving during
$[0,1]$ will be called {\em basic}. Otherwise, we split the $Z$
arriving customers at random into $\lfloor \lambda n \rfloor$ {\em
basic} customers plus the $Z- \lfloor \lambda n \rfloor$ {\em
extra} customers. Let $\tilde{F}$ be the corresponding number of
non-empty queues at time 1 when just the basic customers are
counted. (That is, we pretend that the extra customers never
arrive.) Then
\[\pr(\tilde{F} \geq x|A)= \pr(F \geq
x|Z=\lfloor \lambda n \rfloor) \geq \frac12.\]

Call a queue {\em shunned} if it is empty at time 0, none of the
basic customers has it on their list of $d$ choices (so it stays
empty throughout the interval $[0,1]$ if we ignore any extra
customers), and further there are no potential departures from it
in this period. We shall see that with high probability there are
linearly many such queues. Let $S$ be the number of shunned
queues. Let $0< \eta < (1-\lambda)e^{-(1+d \lambda)}$. We now
prove that
\begin{equation} \label{eqn.shunned}
\pr(S < \eta n) = e^{-\Omega(n^{\frac12})}.
\end{equation}
   Choose $0<\alpha <1-\lambda$ and $0<\beta<\alpha e^{-\lambda d}$
   such that $0< \eta < \beta /e$. Let $B$ be the event that the number
   $n-F_0$ of queues empty at time 0 is at least $\alpha n$.
   Since $\E[F_0] = u(1) n = \lambda n$ (as we noted earlier),
   by Lemma~\ref{lem.conca} we have
   $\pr(\overline{B})= e^{-\Omega(n^{\frac12})}$.
   Let $Q$ consist of the first $\lceil \alpha n \rceil \land (n-F_0)$ empty
queues at time 0, and let $R$ be the set of queues $j \in Q$ such
that no basic customer has $j$ in their list. Then
\[\E[|R| \ | \ B] \ge \lceil \alpha n \rceil
(1-\frac{1}{n})^{d\lfloor \lambda n \rfloor} \sim \alpha
e^{-\lambda d} n.\] By the independent bounded differences
inequality (see for example~\cite{cmcd89}), $\pr(|R| \leq \beta n
\ | \ B)=e^{-\Omega(n)}$. Let $R'$ be the set of the first $\lceil
\beta n\rceil \land |R|$ queues in $R$. For each queue in $R'$,
independently there is no attempted service completion during
$[0,1]$ with probability $e^{-1}$. Then, since $\eta < \beta/e$,
conditional on $|R| \geq \beta n$ with probability
$1-e^{-\Omega(n)}$ at least $\eta n$ of the queues in $R'$ have no
attempted service completion during $[0,1]$, and so are shunned.
Putting the above together yields~(\ref{eqn.shunned}), since
\[\pr(S < \eta n) \leq \pr(S<\eta n\ | \ |R| \geq \beta n) + \pr(|R| <
\beta n\ | \ B) + \pr(\overline{B}).\]

   Let $H$ (for `hit') be the number of shunned queues which are
   the first choice of some extra customer.
   Our next aim is to show~(\ref{eqn.hit}), which says that when
   the event $A$ occurs usually $H$ is large.

   Let ${\cal C}$ be the set of extra customers whose first choice is
   a shunned queue.
   Let $z= \lfloor \frac23 \eta (\lambda n)^{\frac12} \rfloor$.
   Let ${\cal C}'$ consist of the first $|{\cal C}| \land z$
   customers in ${\cal C}$.
   Let $\tilde{Z}$ denote a binomial random variable with parameters
   $\lfloor (\lambda n)^{\frac12} \rfloor$
   and $\eta$. Then
   \[\pr(|{\cal C}'| < z \ | \ S \geq \eta n,A)
   = \pr(|{\cal C}| < z \ | \ S \geq \eta n,A)
   \leq \pr( \tilde{Z} < z)  = e^{-\Omega(n^{\frac12})}.\]
   Now if no shunned queue is first choice for more than two
   customers in ${\cal C}'$, then $H \geq |{\cal C}'| - H'$, where
   $H'$ is the number of shunned queues which are first choice for
   two customers in ${\cal C}'$.  But, given $A$ and $S=s$,
   $\E[H'] \leq {z \choose 2}/s$, and the probability that some
   shunned queue is first choice for more than two customers in ${\cal C}'$
   is at most ${z \choose 3}/s^2$.
   So, setting
   $\delta = \frac13 \eta \lambda^{\frac12}$, it follows from the
   above, using Markov's inequality, that
   \[\pr(H \geq \delta  n^{\frac12} \ | \ S \geq \eta n, A) =1- O(n^{-\frac12})\]
   and so
   \begin{equation} \label{eqn.hit}
   \pr(H \geq \delta  n^{\frac12} \ | \ A) =1- O(n^{-\frac12}).
   \end{equation}

Next, we claim that $F \geq \tilde{F}+H$. To see this, start with
the basic customers, and for each of the $H$ `hit' shunned queues
throw in the first (extra) customer to hit it. With these
customers we have exactly $\tilde{F}+H$ non-empty queues at
time~1. If we now throw in any remaining extra customers, then by
Lemma~\ref{lem.detcoupling} (c) we have $F \geq \tilde{F}+H$ as
claimed. Now
\[\pr(F < x+ \delta n^{\frac12} \ | \ A)
\leq \pr(\tilde{F}<x \ | \ A) + \pr(H < \delta n^{\frac12} \ | \
A) \leq \frac12 +o(1).\]
   Hence,
   \[\pr(F \geq x+ \delta n^{\frac12}) \geq \pr(F \geq x+
   \delta n^{\frac12}|A)\pr(A) \geq \frac1{14} +o(1),\]
   which shows that the variance of $F$ is $\Omega (n)$.
   Thus we have completed the first part of the proof.
   \smallskip

   By expanding
   $\ell (1,Y)^d = [\ell (1) + (\ell(1,Y) - \ell(1))]^d$
   we find
   \begin{eqnarray*}
   0 \le \E[\ell (1,Y)^d]-\ell (1)^d =\sum_{s=2}^d {d \choose s} \E[(\ell
   (1,Y)-\ell (1))^s] \ell (1)^{d-s}.
   \end{eqnarray*}
   If $d \ge 3$, then by the second
   bound in Lemma~\ref{lem.concb} we can upper bound
   \[ | \sum_{s=3}^d {d \choose s}\E[(\ell (1,Y)-\ell (1))^s]\ell (1)^{d-s} | =O(n^{d-3/2}\ln^3 n).\]
   Also, $\E[(\ell (1,Y)-\ell (1))^2] = {\bf\rm{Var}}(F) =\Omega
   (n)$, and so it follows that $\E[\ell (1,Y)^d]-\ell (1)^d = \Omega
   (n^{d-1})$, that is $\E[u(1,Y)^d]-\lambda^d = \Omega (n^{-1})$.
   But by~(\ref{gen0equil}) with $k=1$ (since $u(1)=\lambda$, as we noted before),
   \[ \lambda \E[u(1,Y)^d] -u(2)=0,\]
   and so
   \[\dtv ({\cal L}(Y(1)), {\cal L}_{\lambda,d}) \geq
   u(2) -\lambda^{d+1} = \lambda (\E[u(1,Y)^d]-\lambda^d) = \Omega
   \left( n^{-1}\right),\]
   as required.
   \end{proof}
   \medskip

   %%%%%%%%%%%%%%%%%%%%%%%%%%%

   %It remains here to prove Corollary~\ref{cor.distql}

   \begin{proofof}{Corollary~\ref{cor.distql}}
   Let $k$ be a fixed positive integer.
   Let $X$ denote the random variable $Y^{(n)}(1)$,
   and let the random variable $Z$ have distribution
   ${\cal L}_{\lambda,d}$.
   Let $m_1 = \lfloor 2 \ln\ln n / \ln d \rfloor$.  Then
   $\E[Z^k \ind_{Z \geq m_1}] = e^{- \Omega(\ln^2 n)}$.
   We need a similar bound for $\E[X^k \ind_{X \geq m_1}]$.

   By the bound~(\ref{eqn.ub3}),
   $\pr [ X \geq m_1] = e^{-\Omega( \ln n \ \ln\ln n)}$.
   Also, by Lemma~\ref{lem.march04} (b), for any positive integer $m$
   \[\E[X^k \ind_{X \geq m}] \leq \E[(M^{(n)})^k \ind_{M^{(n)} \geq m}]
   \leq n \sum_{j \geq m} j^k \lambda^j.\]
   For $m$ sufficiently large we have $j^k \lambda^{j/2} \leq 1$
   for all $j \geq m$, and then the last bound is at most
   \[n \sum_{j \geq m} \lambda^{j/2}
   = n \lambda^{m/2} / (1\!-\!\lambda^{\frac12}).\]
   Now let $m_2= \lceil 4 \ln n / \ln \frac{1}{\lambda} \rceil$.  Then
   for $n$ sufficiently large $\E[X^k \ind_{X \geq m_2}]
   %\leq n \sum_{j \geq m_2} \lambda^{j/2}
   %= n \lambda^{m_2/2} / (1\!-\!\lambda^{\frac12})
   \leq e^{-\ln n + O(1)}$.
   Putting the above together we have
   \[ \E[X^k \ind_{X \geq m_1}] \leq m_2^k \pr[ X \geq m_1] + \E[X^k \ind_{X \geq m_2}]
   \leq e^{-\ln n +O(1)}.\]
   Hence
   \begin{eqnarray*}
   |\E[X^k]-\E[Z^k]| & \leq &
   m_1^k \dtv (X,Z) + \E[X^k \ind_{X \geq m_1}] + \E[Z^k \ind_{Z \geq m_1}]\\
   & \leq & e^{-\ln n + O(\ln\ln n)}.
   \end{eqnarray*}

   \end{proofof}

%%%%%%%%%%%%%%%%%%%%%%%%%%%%%%%%%%%%%%%%%%%%%%%%%%%%%%%%%%%%%%%%%%%%
%%%%%%%%%%%%%%%%%%%%%%%%%%%%%%%%%%%%%%%%%%%%%%%%%%%%%%%%%%%%%%%%%%%%

   \subsection{Non-equilibrium case}

   Here we aim to prove Theorem~\ref{cor.distq3}, where we
   still consider a single queue length but we no longer assume
   that the system is in equilibrium.  We shall first prove a
   rather general result, namely Theorem~\ref{thm.distq2};
   and then deduce Theorem~\ref{cor.distq2},
   from which Theorem~\ref{cor.distq3} will follow easily.
   We consider the behaviour of the system from time~0,
   starting from general exchangeable initial conditions.

   Theorem~\ref{thm.distq2} below shows that uniformly over all
   $t \ge 0$ the law of a typical queue length at time $t$ is close to
   ${\cal L}_{v_t}$, where $v_t = (v_t(k): k=0,1,2,\ldots)$ is the
   unique solution to the system~(\ref{eqn.lln}) of differential
   equations subject to the natural initial conditions.
   The upper bound on the total variation distance involves three quantities
   $\delta_n$, $\gamma_n$ and $s_n$ defined in terms of the initial
   distribution for $X_0^{(n)}$.
   Here $\delta_n$ concerns the total number of customers and the
   maximum queue length, $\gamma_n$ concerns concentration of measure, and $s_n$
   concerns how close the proportions of queues of length at
   least $k$ are to the equilibrium proportions.

   \begin{theorem} \label{thm.distq2}
   Let $\theta > 1$.
   There exists a constant $c_2 > 0$ such that for
   any sufficiently large constant $c_1$ and any constant $c_0$
   there is a constant $\varepsilon >0$ for which the following holds.

   Assume that the law of $X^{(n)}_0$ is exchangeable (and thus
   the law of $X^{(n)}_t$ is exchangeable for all $t \ge 0$). Define
   \[\delta_n = n^{-1}\E [\none{X^{(n)}_0} \ind_{\none{X^{(n)}_0}> c_0 n}]
    + \pr(\ninf{X^{(n)}_0} > 2\varepsilon \ln n);\]
   and
   \[\gamma_n = \sup_{f} \pr (|f(X^{(n)}_0) - \E[f(X^{(n)}_0)]|
   \ge (c_1n \ln n)^{1/2} )\]
   where the supremum is over all non-negative functions
   $f \in {\cal F}_1$ bounded above by $n$.

   Let $v^{(n)}_t = (v^{(n)}_t(k): k=0,1,2,\ldots)$ be the unique
   solution to the system~(\ref{eqn.lln}) subject to the initial
   conditions $v^{(n)}_0(k)= \E[u(k,X^{(n)}_0)]$ for $k=0,1,\ldots$.
   Also let $s_n \ge 0$ be given by
   $s^2_n = \sum_{k \ge 1} (v^{(n)}_0(k)-\lambda^{(d^k-1)/(d-1)})^2 \theta^{k}$.
   Then
   \[\sup_{t \ge 0} \dtv ({\cal L}(X^{(n)}_t(1)), {\cal L}_{v^{(n)}_t})
   \leq \varepsilon^{-1}
   \left(n^{-\varepsilon}  + n^{-c_2}s_n + \delta_n + n \gamma_n
   \right).\]
   \end{theorem}
   It is known~\cite{vdk96} that the vector $v^{(n)}_t$
   satisfies~(\ref{decr}) for each $t\geq 0$, and so
   ${\cal L}_{v^{(n)}_t}$ is well-defined.
   Note that the above result is uniform over all positive integers
   $n$, all exchangeable initial distributions of $X^{(n)}_0$, and
   all times $t \geq 0$.
   \medskip

   We shall prove Theorem~\ref{thm.distq2} shortly, but first let us
   give a corresponding result for a particular form of initial
   conditions, which we describe in three steps.
   For each $n$, we start with an initial vector $x$ which is
   not `too extreme' ; then we allow small independent perturbations
   $Z_j$, where we quantify `small' by bounding the moment generating
   function; and finally we perform an independent uniform random
   permutation of the $n$ co-ordinates, in order to ensure exchangeability.
   If say $x={\bf 0}$ and the $Z_j$
   are identically distributed (and non-negative)
   then we may avoid the last `permuting'
   step in forming $X^{(n)}_0$, as that last step is simply
   to ensure that the distribution of $X^{(n)}_0$ is exchangeable.
   \begin{theorem}  \label{cor.distq2}
   For any constant $c_0 \geq 0$ there is a constant $\varepsilon >0$
   such that the following holds.

   Let $\beta >0$.
   Suppose the initial state $X^{(n)}_0$ is obtained as follows. Take
   an integer
   $n$-vector $x$ satisfying $\none{x} \leq c_0 n$ and
   $\ninf{x} \leq \varepsilon \ln n$.
   Let the `perturbations' $Z_1,\ldots,Z_n$ be
   independent random variables each taking
   integer values, each with variance at most $\beta$
   and each satisfying $\E[e^{Z_j/\varepsilon }] \leq \beta$.
   Denote $\tilde{Z}=(\tilde{Z}_1,\ldots,\tilde{Z}_n)$
   where $\tilde{Z}_j= (x(j)+Z_j)^+$.
     Finally let $X^{(n)}_0$ be obtained from $\tilde{Z}$ by
   performing an independent uniform random permutation of the $n$
   co-ordinates.

   Define $v^{(n)}_t$ as in Theorem~\ref{thm.distq2} above.
     Then for each $t \geq 0$, the law of $X^{(n)}_t$ is exchangeable
    and
   \[\sup_{t \ge 0} \dtv ({\cal L}(X^{(n)}_t(1)), {\cal L}_{v^{(n)}_t})
   = O(n^{-\varepsilon}).\]
   \end{theorem}

   Note that if we set $x$ to be the zero vector above then we obtain the simpler
   result we stated earlier as Theorem~\ref{cor.distq3}.
   It remains here to prove Theorems~\ref{thm.distq2} and~\ref{cor.distq2}.

   %%%%%%%%%%%%%%%%%%%%%%%%%%%%%%%%%%%%%%%%%%%%%%%%%%%%%%%%%%%%%%%%%%%%%%%%%%%%%%5
   \smallskip

   \begin{proofof}{Theorem~\ref{thm.distq2}}
   We aim to prove that there is a constant $c_2 >0$
   such that for any constant $c_1$ large enough
     and any constant $c_0$, there is a
   constant $\epsilon >0$ such that
     we have
   \begin{equation} \label{eqn.distrib1}
   \sup_{t \geq 0} \parallel u_t - v_t \parallel_{\infty} =
   O(n^{-\epsilon}+n^{-c_2}s_n + \delta_n+ n \gamma_n),
   \end{equation}
   uniformly over the possible distributions for $X^{(n)}_0$.
   Here $u_t$ refers to the vector $u_t(k), k=0,1,\ldots$, where as
   before $u_t(k) = \E[u(k,X_t)]$.

   To see that the theorem will follow
   from~(\ref{eqn.distrib1}), consider a state $x$ with
   $\ninf{x} \leq m$. Couple $(X_t)$ where $X_0=x$ with $(\tilde{X}_t)$ where
   $\tilde{X}_0={\bf 0}$ in the way described in the previous
   section. By Lemma~\ref{lem.detcoupling},
   we always have
   $\parallel X_t-\tilde{X}_t \parallel_{\infty} \leq m$.  Hence
   always $u(k+m,X_t) \leq u(k,\tilde{X}_t)$, and so
   $\E[u(k+m,X_t)] \leq \E[u(k,\tilde{X}_t)] \leq u(k)$
   (where the last inequality again uses Lemma~\ref{lem.detcoupling}).
   Thus, dropping the condition that $X_0=x$,
   for any non-negative integers $m$ and $k$,
   we have $\E[u(k+m,X_t)|\ninf{X_0} \leq m] \leq u(k)$; and so
   \[ u_t(k+m) =  \E[u(k+m,X_t)] \leq u(k) + \pr(\ninf{X_0} >m).\]
   We shall choose a (small) constant $\epsilon >0$ later,
   and let $m = \lceil 2\epsilon \ln n \rceil$.
   Note that $u(k) =o(n^{-1})$ if $k = \Omega(\ln n)$ by~(\ref{eqn.ub3}),
   and that $\pr (\ninf{X_0} > m) \le \delta_n$. But now we may complete the
   argument as in the proof of Proposition~\ref{prop.distql} (a).

   It remains to prove~(\ref{eqn.distrib1}).
   First we deal with small $t$.
   We begin by showing that there is a constant $\tilde{c} >0$
   such that for each positive integer~$k$, each integer $n \ge 2$,
   each $t>0$ and each $w > 0$,
   \begin{equation} \label{ineql}
   \pr (|\ell (k,X_t)- n u_t(k) | \ge w ) \le
   n e^{-\frac{\tilde{c} w^2}{nt + w}}+
   \pr (|\E[\ell (k,X_t)|X_0]-nu_t(k)| > w/2).
   \end{equation}
   To prove this result,
   note that the left side above can be written as
   \[\sum_{x \in \ints_+^n}
   \pr (|\ell (k,X_t)- n u_t(k) | \ge w \ | \ X_0=x) \Pr(X_0=x).\]
   Let $S= \{x \in \ints_+^n:
|\E[l(k,X_t)|X_0=x] - n u_t(k)| \leq w/2\}$.
   Then by Lemma~\ref{lem.conc2}, %\m{without the $\tilde{X}$'s - ok}
   there is a constant $\tilde{c}$
   such that for each $x \in S$
   \begin{eqnarray*}
   && \pr (|\ell (k,X_t)- n u_t(k) | \ge w \ | \ X_0=x)\\
   & \leq & \pr (|\ell (k,X_t)- \E[l(k,X_t)|X_0=x] | \ge w/2 | \ X_0=x)\\
   & \leq & n e^{-\frac{\tilde{c} w^2}{nt + w}}.
   \end{eqnarray*}
   Also, the contribution from each $x \in \ints_+^n \setminus S$
   is at most $\pr(X_0=x)$, and summing shows that the total
contribution is at most
   \[ \pr(X_0\in \ints_+^n \setminus S) =
   \pr (|\E[\ell (k,X_t)|X_0]-nu_t(k)| > w/2).\]
   This completes the proof of~(\ref{ineql}).

   Now the function $\E[\ell (k,X_t)|X_0=x]$ is a 1-Lipschitz function
   of~$x$ by Lemma~\ref{lem.detcoupling}, and is non-negative
   and bounded above by $n$, so
   \[\pr (|\ell (k,X_t)- n u_t(k) | \ge w ) \le
   n e^{-\frac{\tilde{c} w^2}{nt + w}}+ \sup_f\pr (|f(X_0)-\E[f(X_0)]| > w/2),\]
   where the supremum is over all functions $f \in {\cal F}_1$ such that
   $0 \leq f(x) \leq n$ for each $x \in \ints_+^n$.
   Hence, replacing $w$ by $nw$,
   \begin{eqnarray*}
   \sup_{k \in \nats} \pr(|u(k,X_t)-u_t(k) | \ge w)
   & \le & n e^{-\frac{\tilde{c} w^2 n}{2t}} + n e^{-\frac{\tilde{c} w n}{2}}\\
   & + & \sup_f\pr (|f(X_0)-\E[f(X_0)]| > nw/2),%\\
   \end{eqnarray*}
   since $e^{-\frac{a}{b+c}} \leq e^{-\frac{a}{2b}} +
   e^{-\frac{a}{2c}}$ for $a,b,c>0$.
   Now let $c_1 \ge 2/\tilde{c}$ be a sufficiently large constant and set
   $w= 2\left(c_1 n^{-1} (1+t) \ln n \right)^{\frac12}$.
   Then by the last inequality, we have, uniformly over $t \geq 0$
   \begin{eqnarray*}
   &&\sup_{k \in \nats} \pr (|u(k,X_t)-u_t(k) | \ge w)\\
   & \le & n \ e^{-(4+o(1))\ln n}
   + \sup_f\pr (|f(X_0)-\E[f(X_0)]| \ge \left(c_1 n \ln n \right)^{\frac12})\\
   & \le & o(n^{-2}) +\gamma_n;
   \end{eqnarray*}
that is, uniformly over $t \geq 0$
   \begin{equation} \label{eqn.bound}
   \sup_{k \in \nats} \pr \left( |u(k,X_t)-u_t(k) | \ge
   \left(4c_1 n^{-1} (1+t) \ln n \right)^{\frac12} \right) \le
   o(n^{-2})+\gamma_n.
   \end{equation}
Hence, arguing as in the proof of Lemma 4.2 in~\cite{lm04}
(Lemma~\ref{lem.concb} above), uniformly over $t \ge 0$  \[
\sup_{k \in \nats}\{|\E[u(k,X_t)^d]-u_t(k)^d| \} = O(n^{-1} (1+t)
\ln n +\gamma_n).\]
We shall use this result only for $0 \leq t \leq \ln n$.
By~(\ref{gen0}) we have, uniformly over $0 \leq t \leq \ln n$ and
over $k$,
\[ {{du_t(k)} \over dt} = \lambda (u_t(k-1)^d-u_t(k)^d) -
(u_t(k)-u_t(k+1)) + O( n^{-1} \ln^2 n+\gamma_n).\]
Now recall that $v_t$ is the unique solution to~(\ref{eqn.lln})
subject to the initial conditions $v_0(k)=u_0(k)$ for each $k$.
Also, $2 \lambda d +2$ is a Lipschitz constant of
equation~(\ref{eqn.lln}) under the infinity norm. More precisely,
$2 \lambda d +2$ is a Lipschitz constant of the operator $A$
defined on the space of all vectors $v=(v(k): k \in \nats)$ with
components in $[0,1]$ by
\[ (Av)(k) = \lambda (v(k-1)^d -v(k)^d) -(v(k)-v(k+1)).\]
Equation~(\ref{eqn.lln}) may be expressed succinctly in terms of
$A$ as
\[ \frac{dv_t}{dt} = A v_t.\]
Thus by Gronwall's Lemma (see for instance~\cite{ek}) there exists
a constant $c_3>0$ such that uniformly over $0 \leq t \leq \ln n$
\[ \parallel u_t - v_t \parallel_{\infty} \leq  c_3 (n^{-1} \ln^2
n+\gamma_n)\ln n \ e^{(2\lambda d +2)t}. \] Let $\varepsilon_1 =
\frac{1}{2d+2}$ so $0 < \varepsilon_1 < \frac{1}{2\lambda d+2}$.
Then there exists $\varepsilon_2>0$ such that, uniformly over $t$
with $0 \leq t \leq \varepsilon_1 \ln n$, we have
\begin{equation} \label{eqn.smallt}
\parallel u_t - v_t \parallel_{\infty} = O(n^{-\varepsilon_2}+ n \gamma_n).
\end{equation}

For larger $t$ we introduce the equilibrium distribution into the
argument. Note that $\sup_k |u_t (k) -u (k)| \leq n^{-1} \dw({\cal
L}(X_t),{\bf \Pi})$. Consider Lemma~\ref{thm.statW} with $c=\max\{
c_0, \frac{2\lambda}{1-\lambda} \}$, let $\eta>0$ be as there, and
suppose $2\varepsilon \le \eta \varepsilon_1$. By~(\ref{eqn.ub3}),
there is a constant $c_4>0$ such that $\pr(M \geq \ln\ln n/\ln d
+c_4) = o(n^{-2})$. Hence by Lemma~\ref{thm.statW}, there is a
$\varepsilon_3
>0$ such that, uniformly over $t \ge \varepsilon_1 \ln n$,
\[ \ninf{u_t -u} =o(n^{-\varepsilon_3} + \delta_n).\]
Also, by inequality~(\ref{eqn.utight}) there is a constant $c_5$
such that for each $n$
\[ \sup_k |u(k) - \lambda^{(d^k-1)/(d-1)}| \le c_5 n^{-1}
\ln^2 n.\] Let $s_n(t) \ge 0$ be given by $s^2_n(t) = \sum_{k \ge
1} (v^{(n)}_t(k)-\lambda^{(d^k-1)/(d-1)})^2 \theta^{k}$; thus
$s_n(0)=s_n$. We now use Theorem 2.12 in~\cite{g04}, which says
that for some constants $\gamma_{\theta} > 0$ and
   $C_{\theta} < \infty$,
\[s_n(t)^2 \le e^{-\gamma_{\theta} t} C_{\theta} s_n^2.\]
Hence we deduce that there exists a constant $c_6>0$ such that for
all $t \ge 0$
\[ \sup_k |v_t(k) - \lambda^{(d^k-1)/(d-1)}| \le c_6 e^{-t/c_6} s_n.\]
Combining the last three inequalities, we see that for some
$\varepsilon_4>0$, uniformly over $t \ge \varepsilon_1 \ln n$, we
have
\[ \parallel u_t - v_t \parallel_{\infty} =
O(n^{-\epsilon_4}+n^{-\varepsilon_1/ c_6}s_n+ \delta_n).\] This
result together with~(\ref{eqn.smallt}) completes the proof
of~(\ref{eqn.distrib1}) (with $c_2 = \varepsilon_1/c_6$), and thus
of the whole result.
\end{proofof}

\begin{proofof}{Theorem~\ref{cor.distq2}}
Clearly $X_0$ is exchangeable, and thus by symmetry so is $X_t$
for each $t \geq 0$. Let us ignore the last step in constructing
the initial state, which involves the random permutation. Note
that this does not affect $u_t = (u_t(k): k=0,1,2,\ldots)$, and
similarly $v_0(k) = \E[u(k,X_0)] = \E[u(k,\tilde{Z})]$.

   We shall choose a small constant $\varepsilon > 0$ later,
   and assume that $\E[e^{Z_j/\varepsilon }] \leq \beta$.
   By Markov's inequality,
   for each $n$ and each $b \geq 0$,
   \begin{equation} \label{eqn.maxzj}
   \pr(\max_j Z_j \geq b) \leq
   \sum_j e^{- b/\varepsilon } \E[e^{Z_j/\varepsilon}]
   \leq n \beta e^{- b/\varepsilon }.
   \end{equation}
   Let us take $b=b(n)=\ln^2n$, and let $A=A(n)$ be the event that
   $\max_j Z_j \leq b$.  Then by~(\ref{eqn.maxzj}) we have
   $\pr(\bar{A}) = O(e^{-\ln^2 n})$.
   Given a function $f \in {\cal F}_1$ such that
   $0 \leq f(x) \le n$ for all $x \in \ints_+^n$,
   let $\mu = \E [f(X_0)]$ and let
   $\tilde{\mu }= \E[f(X_0)|A]$.
   Then uniformly over such $f$
   \[| \tilde{\mu}- \mu | =
   |\tilde{\mu} - \E[f(X_0)|\bar{A}]|  \pr(\bar{A})
   \leq n \pr(\bar{A}) = O(e^{-\ln^2 n}).\]
   Also, the variance of $Z_j$ given $Z_j \leq b$ is at most
   $var(Z_j)/\pr(Z_j \leq b)$, which is at most $2 \beta$
   for $n$ sufficiently large that $\pr(Z_j \leq b) \geq \frac12$.
   Thus given $Z_j \leq b$, the variance of $\tilde{Z}_j$ is also at most
   $2 \beta$, and further $\tilde{Z}_j$ has range at most $\frac32 b$,
   assuming that $2 \varepsilon \leq \ln n$.
   Hence, by Lemma~\ref{lem.bernstein}, for each sufficiently large
   integer $n$
   and each $w > 0$,
   \[\pr (|f(X_0)- \tilde{\mu} | \ge w | A)
   \le 2 e^{-\frac{w^2}{4n\beta + bw}},\]
   and so the quantity $n\gamma_n$ in Theorem~\ref{thm.distq2} is
   $O(n^{-1})$
   provided the constant $c_1$ is large enough.

   Note that $\E[Z_j] \leq \varepsilon \beta$ for each $j$.  We
   shall choose $0<\varepsilon \leq 1$ below.
   Let $c=\max \{c_0+\beta +1, \frac{1}{1-\lambda}\}$. Then
   $\E[\none{X_0} {\bf 1}_{\none{X_0}>c n}] = o(1)$ since
   $c > c_0 + \beta$. Now assume that $0<\varepsilon \leq 1$ is
   such that the bound in Theorem~\ref{thm.distq2}
   holds with this value of $\varepsilon$ and with $c_0=c$.
   Then by~(\ref{eqn.maxzj}), with $b= 2 \varepsilon \ln n$,
   \[ \pr(\ninf{X_0} > 2 \varepsilon \ln n) \leq n \beta e^{-2 \ln n} = O(n^{-1}).\]
   Thus the quantity $\delta_n$ in Theorem~\ref{thm.distq2} is $O(n^{-1})$.

   It remains to upper bound the term $s_n$.
   Note that
   \[s_n \le  \sum_{k \ge 1} \theta^{k}({(v_0^{(n)}(k))}^2
   + \sum_{k \ge 1} \theta^{k}\lambda^{2(d^k-1)/(d-1)}.\]
   The second of these terms certainly takes a finite value independent of $n$,
   so it suffices to ensure that the first term is
   $O(n^{c_2/2})$ where the constant $c_2>0$ is from
   Theorem~\ref{thm.distq2}.
   Now
   $v_0^{(n)}(k) \leq u(\lfloor k/2\rfloor,x) + \max_j
   \Pr(Z_j \geq k/2)$, and so
   ${v_0^{(n)}(k)}^2 \leq 2 u(\lfloor
   k/2\rfloor,x)^2 + 2\max_j \Pr(Z_j \geq k/2)^2$.
   Hence
   \begin{eqnarray*}
   \sum_{k \ge 1} \theta^{k}({(v_0^{(n)}(k))}^2
   & \le &  2\sum_{k \ge 1} \theta^{k}{(u(\lfloor k/2 \rfloor,x))}^2 \\
   & +  & 2\sum_{k \ge 1} \theta^{k}{(\max_j \pr (Z_j \ge k/2))}^2.
   \end{eqnarray*}
   Using Markov's inequality again,
   $\pr (Z_j \ge k/2) \leq \beta e^{-k/2\varepsilon}$, and so
   the second term is bounded by a constant uniformly in $n$, provided
   $\theta < e^{1/\varepsilon}$. This is certainly true if
   $\varepsilon > 0$ is small enough.
   As for the first term, if we let $k_0 = 2 \varepsilon \ln n +2$,
   we have
   \begin{eqnarray*}
   \sum_{k \ge 1} \theta^{k}{(u(\lfloor k/2 \rfloor,x))}^2
   & \le & \sum_{1 \le k \le k_0} \theta^{k}
   {(u(\lfloor k/2 \rfloor,x))}^2\\
   & \le & k_0 \ \theta^{k_0}
   \;\; \leq n^{c_2/2}
   \end{eqnarray*}
   if $\varepsilon$ is sufficiently small.
Combining the above estimates, we obtain the required bound on
$s_n$.
   \end{proofof}

%%%%%%%%%%%%%%%%%%%%%%%%%%%%%%%%%%%%%%%%%%%%%%%%%%%%%%%%%%%%%%%%%%%%%%%%%%%%%%%%
%%%%%%%%%%%%%%%%%%%%%%%%%%%%%%%%%%%%%%%%%%%%%%%%%%%%%%%%%%%%%%%%%%%%%%%%%%%%%%%%%%

   \section{Asymptotic chaos}
   \label{sec.chaos}

   \subsection{Equilibrium case}

     In this subsection we prove Theorem~\ref{thm.chaos1}
     concerning chaoticity of the queue-lengths process,
     where the system is stationary.  This result quantifies the
     chaoticity in terms of the total variation distance between the
     joint law of the queue lengths and the corresponding product laws.
     In the proof we first bound another natural measure of these
     distances, following the treatment in~\cite{s91}
     -- see~(\ref{eq.dtv}) and~(\ref{eq.dtv-1}) below.

    \begin{proofof}{Theorem~\ref{thm.chaos1}}
    By Lemma~\ref{lem.conca}, there exists a constant $c> 0$ such that
    \[\sup_{f \in {\cal F}_1} \pr \left(|f(Y)-\E f(Y)| \ge cn^{1/2} \ln n \right) =
    o(n^{-2}).\]
    Let ${\cal F}_1'$ be the set of functions $f \in {\cal F}_1$ with
    $\ninf{f} \leq n$.
    By the last result, uniformly over positive integers $a$,
    \begin{equation} \label{fconc}
    \sup_{f_1,\ldots,f_a \in {\cal F}_1'}
    \E \left[\prod_{s=1}^a | f_s(Y) - \E[f_s(Y)]| \right] \leq (cn^{1/2} \ln
    n)^a +o(an^{a-2}).
    \end{equation}
   Let ${\cal G}_1$ denote the set of measurable real-valued functions
   $\phi$ on $\ints_+$ with $\ninf{\phi} \le 1$.
   For any measurable real-valued function $\phi$ on $\ints_+$ let
   $\bar{\phi}$ be the function on $\ints_+^n$ defined by setting
   $\bar{\phi}(y)= \frac{1}{n} \sum_{i=1}^{n} \phi(y_i)$ for
   $y=(y_1,\ldots,y_n) \in \ints_+^n$.
   Observe that
   \[ \bar{\phi}(y) = \sum_{k=0}^{\infty} \phi(k)
   (u(k,y)-u(k+1,y)) = \phi(0)+ \sum_{k=1}^{\infty} (\phi(k)-\phi(k-1))
   u(k,y).\]
   Observe also that since the distribution of $Y$ is exchangeable we have
   $\E[\bar{\phi}(Y)] = \E[\phi(Y(1))]$:
   let us call this mean value $\hat{\phi}$.

   Let $f(y) = (n/2) \bar{\phi}(y) = \frac12 \sum_{i=1}^{n}
   \phi(y_i)$.
   It is easy to see that if $\phi \in {\cal G}_1$ then $f \in {\cal F}_1'$.
   Hence by~(\ref{fconc}), uniformly over positive integers $a$,
   \begin{equation} \label{eq.chaos.1}
   \sup_{\phi_1,\ldots, \phi_a \in {\cal G}_1} \E\left[\prod_{s=1}^a
   |\bar{\phi_s}(Y) -\hat{\phi_s}| \right] \leq (2c n^{-1/2} \ln n)^{a}
   + o(a 2^a n^{-2}) \leq (3c n^{-1/2} \ln n)^{a}
   \end{equation}
   for $n$ sufficiently large.
   But, by writing $\bar{\phi_s}(Y)$ as
   $(\bar{\phi_s}(Y) - \hat{\phi_s}) + \hat{\phi_s}$, we see that
   \[\E\left[\prod_{s=1}^a \bar{\phi_s}(Y)\right]-\prod_{s=1}^a \hat{\phi_s}\]
   can be expanded as
   \[\sum_{A \subseteq \{1, \ldots, a\}, |A| \ge 2}
   \E \left[\prod_{s \in A} (\bar{\phi_s}(Y)-\hat{\phi_s})\right]
   \prod_{s \in \{1, \ldots , a\} \setminus A } \hat{\phi_s}.\]
   Hence by~(\ref{eq.chaos.1}), uniformly over all positive integers
   $r \leq n^{\frac12}/ \ln n$
   \begin{eqnarray} \label{eq.chaos.2}
   && \sup_{\phi_1, \ldots, \phi_r \in {\cal G}_1}
   |\E\left[\prod_{s=1}^r \bar{\phi_s}(Y)\right]-\prod_{s=1}^r \hat{\phi_s}| \nonumber\\
   \nonumber
   & \le & \sum_{A \subseteq \{1, \ldots, r\}, |A| \ge 2}
   \E \left[\prod_{s \in A} |\bar{\phi_s}(Y)- \hat{\phi_s}|\right]\\
   & \leq & \sum_{a=2}^{r} {r \choose a} (3cn^{-\frac12} \ln n)^a
   \leq \sum_{a=2}^{\infty} (\frac{3ec}{a} \ rn^{-\frac12} \ln n)^a \nonumber\\
   & = & O(r^2 n^{-1} \ln^2 n).
   \end{eqnarray}
   Now
   \[\E\left[\prod_{s=1}^r \bar{\phi_s}(Y)\right]=
   n^{-r} \E \left[\prod_{s=1}^r \sum_{j=1}^n \phi_s (Y(j))\right]
   = \E\left[\prod_{s=1}^r \phi_s (Y(s))\right] +O(r^2 n^{-1})\]
   uniformly over $r$ and all $r$-tuples $\phi_1, \ldots,\phi_r \in {\cal G}_1$
   since when we expand the middle expression there are at most
   $r^2 n^{r-1}$ terms for which the values of $j$ are not all
   distinct. Hence from~(\ref{eq.chaos.2}), uniformly over all
   positive integers $r \leq n$,
   \begin{equation} \label{eq.dtv}
   \sup_{\phi_1, \ldots,\phi_r \in {\cal G}_1}
   | \E [\prod_{s=1}^r \phi_s (Y(s))]
   - \prod_{s=1}^r \hat{\phi_s} | =O(r^2 n^{-1}\ln^2 n).
   \end{equation}
   (For $r > n^{\frac12}/ \ln n$ the result is trivial.)
   Further, by Proposition~\ref{prop.distql}(a),
   \[ \sup_{\phi \in {\cal G}_1}
   |\hat{\phi} - \E_{{\cal L}_{\lambda,d}} (\phi)|= O(n^{-1}
   \ln^2 n \ \ln \ln n),\]
   where $\E_{{\cal L}_{\lambda,d}} (\phi)$ denotes the expectation with
   respect to the limiting measure ${\cal L}_{\lambda,d}$.
   Thus
   \[ \sup_{\phi_1, \ldots,\phi_r \in {\cal G}_1}
   | \prod_{s=1}^r \hat{\phi}_s
   - \prod_{s=1}^r \E_{{\cal L}_{\lambda,d}} (\phi_s) |
   =O(r n^{-1} \ln^2 n \ \ln \ln n)\]
   uniformly over all positive integers $r$
   (again noting that the result is trivial for large $r$).
   It follows that, uniformly over all positive integers $r \leq n$
   \begin{equation} \label{eq.dtv-1}
   \sup_{\phi_1, \ldots,\phi_r \in {\cal G}_1}
   | \E [\prod_{s=1}^r \phi_s (Y(s))]
   - \prod_{s=1}^r \E_{{\cal L}_{\lambda,d}} (\phi_s) |
   =O(n^{-1}\ln^2 n \ (r^2+r\ln \ln n)).
   \end{equation}

   Following Sznitman~\cite{s91}, the vector $Y=(Y(1),\ldots,Y(n))$ is chaotic (in
   total variation) since the left hand side of~(\ref{eq.dtv}) tends to
   0 as $n \to \infty$ for each fixed positive integer $r$.
   Thus~(\ref{eq.dtv}) quantifies the chaoticity of the equilibrium queue lengths
   $Y$ in terms of this definition.
   Similarly, the inequality~(\ref{eq.dtv-1}) quantifies $Y$ being ${\cal
   L}_{\lambda,d}$-chaotic.
   Let us note that, up to factors logarithmic in $n$, the bound in~(\ref{eq.dtv-1}) is
   of the same order as the time-dependent bound of Theorem 4.1
   in~\cite{gm94} obtained for a related class of models.

   The results~(\ref{eq.dtv}) and~(\ref{eq.dtv-1}) yield bounds
   on the total variation distance between the joint law of
   $Y(1),\ldots,Y(r)$ and the product law ${\cal L}(Y(1))^{\otimes r}$,
   and between the joint law of $Y(1),\ldots,Y(r)$ and the product law
   ${\cal L}_{\lambda,d}^{\otimes r}$ respectively, as follows.
   [In general, even for random variables $Y(s)$ taking
   values in $\ints_+$ and $r=2$, it does not follow that if the left hand side
   of~(\ref{eq.dtv}) tends to 0 as $n \to \infty$ then
   the total variation distance between the joint law of
   $Y(1),\ldots , Y(s)$ and the product law ${\cal L}(Y(1))^{\otimes r}$
   tends to 0.]
   Putting $\phi_s$
   as the indicator of the set $\{k_s\}$,
     we obtain, uniformly over positive integers $r \leq n$
  \begin{equation} \label{eq.chaos.3}
   \sup_{k_1,\ldots,k_r} \left| \pr(\land_{s=1}^r (Y(s)=k_s))
   -\prod_{s=1}^r \pr (Y(s)=k_s)\right| = O(r^2 n^{-1}\ln^2 n )
   \end{equation}
   and
  \begin{equation} \label{eq.chaos.4}
   \sup_{k_1,\ldots,k_r} \left| \pr(\land_{s=1}^r (Y(s)=k_s))
   -\prod_{s=1}^r {\cal L}_{\lambda,d}(\{k_s\})\right|
   = O(n^{-1}\ln^2 n \ (r^2 + r \ln \ln n) ).
   \end{equation}
   But by~(\ref{eqn.ub3}), there exists a constant $c>0$ such that
   \[ \pr (\max_j Y(j) > {\ln \ln n}/ {\ln d} +c ) =O(n^{-1}).\]
   Hence in total variation distance, for any fixed $\epsilon > 0$,
     the joint law of $Y(1), \ldots, Y(r)$ differs from the product law
   ${\cal L}(Y(1))^{\otimes r}$
   by at most $O(n^{-1}\ln^2 n \ ((1+\epsilon) \ln \ln n/ \ln d)^r )$, and by at
   most $O(n^{-1}\ln^2 n \ ((1+\epsilon) \ln \ln n/ \ln d)^{r+1} )$
   from the law ${\cal L}_{\lambda,d}^{\otimes r}$.
   \end{proofof}

%%%%%%%%%%%%%%%%%%%%%%%%%%%%%%%%%%%%%%%%%%%%%%%%%%%%%%%%%%%%%
%%%%%%%%%%%%%%%%%%%%%%%%%%%%%%%%%%%%%%%%%%%%%%%%%%%%%%%%%%%%%

  \subsection{Non-equilibrium case}

   We now no longer assume that the system is in equilibrium,
   and show that under quite general exchangeable initial
   conditions, chaotic behaviour occurs in the system,
   uniformly for all times.
   We need to prove Theorem~\ref{cor.chaos2}.
   We first state two general results,
   Theorem~\ref{thm.chaos2} and~\ref{cor.chaos1}, and note that
   Theorem~\ref{cor.chaos2} will follow easily from the latter.
   We then prove Theorem~\ref{thm.chaos2} and deduce Theorem~\ref{cor.chaos1}.

   Given an $n$-queue process $(X_t^{(n)})$ and a positive integer $r \leq n$,
   let ${\cal L}_t^{(n,r)}$ denote the joint law of
   $X_t^{(n)}(1),\ldots, X_t^{(n)}(r)$, and let $\tilde{\cal L}_t^{(n,r)}$
   denote the product law of $r$ independent copies of $X_t^{(n)}(1)$.
   The following result shows that, as long as initially there are
   not too many customers in the system, the maximum queue is not too
   long and the system has sufficient concentration of measure, there
   will be chaotic behaviour uniformly for all times.
   \begin{theorem} \label{thm.chaos2}
   For each sufficiently large constant $c_1$ and each constant
   $c_0$, there is a constant $c$ such that the following holds.

   Let $n$ be a positive integer.
   Assume that the law of  $X_0^{(n)}$ is exchangeable
   (which implies that the law of $X_t^{(n)}$ is exchangeable for all $t
   \geq 0$).
   Let $m \geq 0$,
   let
   \[\delta_{n,m} = \pr (\none{X_0^{(n)}}>c_0 n ) + \pr ( \ninf{X_0^{(n)}} > m),\]
   and let
   \[\gamma_n = \sup_{f} \pr (|f(X^{(n)}_0) - \E[f(X^{(n)}_0)]|
   \ge (c_1 n \ln n)^{1/2} ),\]
   where the supremum is over all non-negative functions $f \in {\cal
   F}_1$
   bounded above by $n$.
   Then for each positive integer $r \leq n$ we have
   \[ \sup_{t \ge 0} \dtv \left( {\cal L}_t^{(n,r)}, \tilde{\cal L}_t^{(n,r)} \right)
   \leq c [ (n^{-1}\ln n)(m+\ln n)+ \gamma_n] (m+2\ln\ln n)^r+(r+1)\delta_{n,m}.\]
   \end{theorem}
   We shall see later that it is
   straightforward to deduce Theorem~\ref{thm.chaos1} from
   Theorem~\ref{thm.chaos2}.
   Theorem~\ref{cor.chaos1} below
   is a straightforward consequence of
   Theorem~\ref{thm.chaos2} and shows that, in particular,
   there is chaotic behaviour
   uniformly for all times $t \ge 0$,
   when the initial state is obtained by perturbing a ``nice'' queue-lengths
   vector by a set of `small' independent random variables.
   \begin{theorem} \label{cor.chaos1}
   Let $c_0>0$, $s_0>0$ and $\alpha >0$ be any constants.
   Then there is a constant $c>0$ such that the following holds.

   Suppose that the initial
   state $X^{(n)}_0$ is obtained as follows.
   Take $x \in \ints_+^n$ with $\none{x} \leq c_0 n$ and let
   $m = \ninf{x}$. Let the `perturbations'
   $Z_1,\ldots,Z_n$ be independent random variables each taking
   integer values, each with variance at most $\alpha$ and each
   satisfying $\E[e^{s_0 Z_j}] \leq \alpha$.
   Let $\tilde{Z}=(\tilde{Z}_1,\ldots,\tilde{Z}_n)$
   where $\tilde{Z}_j= (x(j)+Z_j)^+$.
   Finally let $X^{(n)}_0$ be obtained from $\tilde{Z}$ by
   performing an independent uniform random permutation of the $n$
   co-ordinates.

   Then for each $t \ge 0$, $X^{(n)}_t$ is exchangeable, and
   for each positive integer $r \leq n$ we have
   we have
   \begin{eqnarray*}
   \sup_{t \ge 0} \dtv
   \left( {\cal L}_t^{(n,r)}, \tilde{\cal L}_t^{(n,r)} \right)
   & \leq & c n^{-1} (m+ 2 \ln n)^{r+2}.
   \end{eqnarray*}
   \end{theorem}
   Note that Theorem~\ref{cor.chaos2} will follow from this last
   result on taking $x$ as the zero vector.  It remains to prove
   Theorems~\ref{thm.chaos2} and~\ref{cor.chaos1}.
   \medskip

%%%%%%%%%%%%%%%%%%%%%%%%%%%%%%%%%%%%%%%%%%%%%%%%%%%%%%%%%%%%%%%%%%%%%
%%%%%%%%%%%%%%%%%%%%%%%%%%%%%%%%%%%%%%%%%%%%%%%%%%%%%%%%%%%%%%%%%%%%%

   \begin{proofof}{Theorem~\ref{thm.chaos2}}
   First we deal with small $t$.
   We start by arguing in a similar way to the proofs of
   Theorems~\ref{thm.distq2} and~\ref{thm.chaos1}.
   As in the inequality~(\ref{eqn.bound}),
   there exists a constant $c_1>0$ such that uniformly over $t \geq 0$,
   over $n$, and over all non-negative functions $f \in {\cal F}_1$
   bounded above by $n$,
   we have
   \[ \pr (| n^{-1}f(X_t) -\E [n^{-1}f(X_t)]| \ge (c_1 (1+t) n^{-1} \ln n)^{1/2})
   \leq o(n^{-2}) + \gamma_n.\]
   With notation as in the proof of Theorem~\ref{thm.chaos1},
   since the distribution of $X^{(n)}_t$ is exchangeable, we have
   $\E[\bar{\phi}(X^{(n)}_t)] = \E[\phi(X^{(n)}_t(1))]$: let us call this mean value
   $(\hat{\phi})_t$.
   Uniformly over integers $a$ and $r$
   with $2 \leq a \le r$,
   \[ \sup_{\phi_1, \ldots , \phi_a \in {\cal G}_1} \left [
   \E\prod_{s=1}^a |\bar{\phi}_s(X_t^{(n)})-(\hat{\phi}_s)_t | \right]
   \leq  O( (1+t) n^{-1} \ln n) + r \gamma_n,\]
   Using the above bound and following the steps in the proof of
   Theorem~\ref{thm.chaos1}, we find that uniformly over $r$
   \[ \sup_{\phi_1, \ldots , \phi_r \in {\cal G}_1}
   \left| \E \left [\prod_{s=1}^r \phi_s (X^{(n)}_t(s)) \right ]
   -\prod_{s=1}^r (\hat{\phi_s})_t \right|\]
   \[ = O\left (r^2 [(1+t) n^{-1}\ln n +  r \gamma_n ] \right ).\]

   Let $M^{(n)}$ denote the maximum queue length in equilibrium.
   Then, arguing as in the proof of Theorem~\ref{thm.distq2}, given
   that $\ninf{X^{(n)}_0} \le m$,
   we have by Lemma~\ref{lem.detcoupling} and~(\ref{eqn.ub3}) that
   \[ \pr (\ninf{X_t^{(n)}} \geq m+k) \leq \pr(M^{(n)} \geq k) = O(n^{-1})\]
   if $k = \lfloor \ln\ln n / \ln d \rfloor + c_2$ for some suitable constant
   $c_2$.
   Hence, uniformly over $t \geq 0$, over $n$,
   over positive integers $2 \leq a \le r$,  over $m \geq 0$ and
   over exchangeable distributions for $X_0^{(n)}$, we have
   \begin{equation} \label{eq.chaos.5}
   \dtv \left( {\cal L}_t^{(n,r)}, \tilde{\cal L}_t^{(n,r)} \right)
   = O \left( [(1+t)  n^{-1} \ln n + \gamma_n ] (m+ 2\ln \ln n)^r \right)+\delta_{n,m}.
   \end{equation}
   The above takes care of all $t = O(m+\ln n)$.
   More precisely, we use the bound~(\ref{eq.chaos.5}) for
   $0 \leq t \leq \eta^{-1}(m+2 \ln n)$,
   where $\eta > 0$ is as in Lemma~\ref{thm.stat} with $c=c_0$.

   To handle larger $t$ we introduce the equilibrium distribution.
   Let $Y^{(n)}$ denote a queue-lengths vector in equilibrium.
   Denote by ${\cal M}^{(n,r)}$ the joint law of $Y^{(n)}(1),
   \ldots,Y^{(n)}(r)$; and denote by $\tilde{\cal M}^{(n,r)}$ the
   product law of $r$ independent copies of $Y^{(n)}(1)$.
   Then for $t \ge \eta^{-1}(m+2 \ln n)$ we have by Lemma~\ref{thm.stat}
   and Theorem~\ref{thm.chaos1} that
    \begin{eqnarray*}
   && \dtv ({\cal L}_t^{n,r},\tilde{\cal L}_t^{n,r})\\
   & \leq & \dtv ({\cal L}_t^{n,r},{\cal M}_t^{n,r})
   + \dtv ({\cal M}_t^{n,r},\tilde{\cal M}_t^{n,r})
   + \dtv (\tilde{\cal M}_t^{n,r},\tilde{\cal L}_t^{n,r})\\
   & \leq &  \dtv ({\cal M}_t^{n,r},\tilde{\cal M}_t^{n,r})
   + (r+1) \dtv ({\cal L}(X_t), {\cal L}(Y)) \\
   & \leq &  O(n^{-1} \ln^2 n(2 \ln \ln n)^r ) + (r+1)\delta_{n,m}.
   \end{eqnarray*}
   This completes the proof of Theorem~\ref{thm.chaos2}.
   \end{proofof}

   Let us indicate briefly how to deduce Theorem~\ref{thm.chaos1} from
   Theorem~\ref{thm.chaos2}. By Lemma~\ref{lem.march04} and
   equation~(\ref{eqn.ub3}), if we choose a constant
   $c_0 > \frac{\lambda}{1-\lambda}$ and
   $m = \ln \ln n /\ln d + \tilde{c}$
   for a sufficiently large constant $\tilde{c}$ then
   the term $\delta_{n,m}$ in Theorem~\ref{thm.chaos2} is $O(n^{-1})$.
   Further, by Lemma~\ref{lem.conca}, if the constant $c_1$ is
   sufficiently large then the term $\gamma_n$ is also $O(n^{-1})$,
   and Theorem~\ref{thm.chaos1} follows.

   \smallskip

   \begin{proofof}{Theorem~\ref{cor.chaos1}}
   We fix $s_0>0$ and $\alpha > 0$ such that
   $\E[e^{s_0 Z_j}]\le \alpha$ for each $j$.
   Let $\beta = \ln (1+ \alpha)$, so that $\beta >0$.
   Note that
   \[\none{X^{(n)}_0} \; = \; \none{\tilde{Z}} \; \leq \; \none{x} +
   \sum_{j=1}^{n} Z_j^+\]
   and
   \[ \E [e^{s_0 Z_j^+}] \leq 1 + \E [e^{s_0 Z_j}] \leq (1+\alpha) =
   e^{\beta}.\]
   Thus
   \begin{eqnarray*}
   \pr \left (\none{X_0} \geq (c_0 + \frac{2\beta}{s_0}) n \right )
   & \leq &
   \pr(\sum_{j=1}^{n} Z_j^+ \geq \frac{2\beta}{s_0} n)\\
   & \leq & e^{-2\beta n} \E[e^{s_0\sum_{j=1}^n Z_j^+}] \\
   & = & e^{-2\beta n} \prod_{j=1}^n \E[e^{s_0 Z_j^+ }] \; \le \; e^{-\beta n}.
   \end{eqnarray*}
   Also,
   \begin{eqnarray*}
   \pr \left (\ninf{X_0} > (m+\frac{2}{s_0} \ln n) \right )
   & \leq & \sum_{j=1}^n \pr (Z_j > \frac{2}{s_0} \ln n) \\
   & \leq  & e^{-2 \ln n} \sum_{j=1}^n \E[e^{s_0 Z_j}] = O(n^{-1}).
   \end{eqnarray*}
   Further, arguing as in the proof of Theorem~\ref{cor.distq2}, by
   Lemma~\ref{lem.bernstein}  the constant $c_1$ can
   be chosen large enough so that the term $\gamma_n$ in
   Theorem~\ref{thm.chaos2} is $O(n^{-1})$.
   Hence the result follows from
   Theorem~\ref{thm.chaos2} with this value of $c_1$, with $c_0$
   replaced by $c_0 +  \frac{2\beta}{s_0}$ and with $m$ replaced by $m +
   \frac{2}{s_0} \ln n$.
   We have
   \begin{eqnarray*}
   \sup_{t \ge 0} \dtv
   \left( {\cal L}_t^{(n,r)}, \tilde{\cal L}_t^{(n,r)} \right)
   & \leq &
   c n^{-1} \ln n (m+\ln n) (m + \ln n + 2 \ln \ln n)^r\\
   & \leq & c n^{-1} (m+ 2 \ln n)^{r+2}.
   \end{eqnarray*}

   \end{proofof}

   \smallskip

\noindent {\bf Acknowledgement} We would like to thank the referee
for detailed and helpful comments.

%%%%%%%%%%%%%%%%%%%%%%%%%%%%%%%%%%%%%%%%%%%%%%%%%%%%%%%%%%%%%%%%%%%%%%%%%%%%
%%%%%%%%%%%%%%%%%%%%%%%%%%%%%%%%%%%%%%%%%%%%%%%%%%%%%%%%%%%%%%%%%%%%%%%%%%%%%

\end{document}